\documentclass[a4paper,oneside,11pt]{scrartcl}
\input{mypaper.sty}
\def\csat {c_{\textnormal{sat}}}
\def\csp {c_{\textnormal{sp}}}

\usepackage{bm}

\begin{document}
\title{Block-adaptive Cross Approximation of Discrete Integral Operators}
\author{M. Bauer and M. Bebendorf\footnote{Mathematisches Institut, Universit\"at Bayreuth, 95447 Bayreuth, Germany}}
\date{\today}
\maketitle

\begin{abstract}
In this article we extend the adaptive cross approximation (ACA) method known for the efficient approximation of discretisations of
integral operators to a block-adaptive version. While ACA is usually employed to assemble hierarchical matrix approximations
having the same prescribed accuracy on all blocks of the partition, for the solution of linear systems it may be more
efficient to adapt the accuracy of each block to the actual error of the solution as some blocks may be more
important for the solution error than others. To this end, error estimation techniques known from adaptive mesh refinement are
applied to automatically improve the block-wise matrix approximation. This allows to interlace the assembling of the coefficient matrix
with the iterative solution.
\end{abstract}
\textbf{Keywords:} ACA, error estimators, BEM, non-local operators, hierarchical matrices, fast solvers

\section{Introduction}
Methods for the solution of boundary value problems of elliptic partial differential equations typically involve non-local operators.
Examples are fractional diffusion problems (see~\cite{ag2017, npv12, dg13}) and boundary integral methods (see~\cite{SaSc11, os08}),
in which amenable boundary value problems are reformulated as integral equations over the boundary of the computational domain.
The discretisation of such non-local operators usually leads to fully populated matrices. The corresponding systems of linear equations
are quite expensive to solve numerically in the typical case of large numbers of unknowns. Additionally, assembling and storing the discrete operators may
be quite costly. 
Even if the formulation of the boundary value problem is local and its discretisation (e.g.\ by the finite element method) leads to a sparse matrix,
it may be useful to construct non-local approximations (e.g.\ for preconditioning) to its inverse or to triangular factorisations.
In both cases, suitable techniques have to be employed if the boundary value problem is to be solved with linear or almost linear
complexity.

There are several ways how fully populated discretisations of elliptic operators can be treated efficiently.
In addition to \textit{fast multipole methods}~\cite{gr87} or \textit{mosaic decomposition methods}~\cite{t96},
\textit{hierarchical matrices} ($\mathcal{H}$-matrices)~\cite{hackbusch99, hk00} have been introduced for the asymptotic optimal numerical
treatment. This kind of methods can be regarded to be based on a hierarchical partition of the discretised operator $A \in \R^{N \times N}$ 
into suitable blocks combined with a local low-rank approximation. Hence, fully populated matrices can be approximated by data-sparse
representations which require only a logarithmic-linear amount of storage. Furthermore, these representations can be applied to vectors with logarithmic-linear complexity.
In addition, $\Hm$-matrices provide approximate substitutes for the usual matrix operations including the LU~factorisation
with logarithmic-linear complexity; see~\cite{bebendorf08}.
These properties strongly accelerate iterative solvers like the \textit{conjugate gradient}~(CG) method or GMRES (see~\cite{saad03}) for the solution of the
arising linear systems of equations.

A crucial step for the efficient treatment of fully populated matrices $A$ is their assembling. While explicit kernel approximations were
used in the early stages of fast methods for the treatment of integral operators, in recent years the \textit{adaptive cross approximation}~(ACA) (see~\cite{bebendorf00, br03})
has become quite popular. The latter method employs only few of the original matrix entries of~$A$ for its data-sparse
approximation by hierarchical matrices. Since one is usually interested in approximations $\tilde A$ of~$A$ leading to a sufficiently
accurate result when $\tilde A$ or its inverse is applied to arbitrary vectors, it is common to
approximate all blocks of the matrix partition uniformly with a prescribed accuracy.
The universality of the approximation to the discrete operator~$A$ comes at a high price due to the generation of information that may be redundant for the solution of the linear
system.
When the quantity of interest is the error of the solution~$x$ of the linear system $Ax=b$, $\tilde A$ obtained from usual ACA may not be the ideal
approximation to~$A$. In such a situation, certain parts of the discrete operator~$A$ may be more important than others. Conversely,
depending on the right-hand side~$b$ some blocks may not be that important for the error of~$x$.

In this article we propose to construct the hierarchical matrix approximation in a block-adaptive
way, i.e., in addition to the adaptivity of ACA on each block we add another level of adaptivity to the
construction of the hierarchical matrix. This variant of ACA will be referred to as BACA.
In order to find an $\Hm$-matrix approximation that is better suited for the solution error,
we employ techniques known from adaptivity together with suitable error estimators.
The strategy of error estimation is well known in the context of numerical methods for partial differential or integral equations. 
 Adaptive methods usually focus on mesh refinement. 
Here, these strategies will be used to successively improve
block-wise low-rank approximations, whereas the underlying grid (and the hierarchical block structure) does not change.
Although the ideas of this approach are presented in connection with ACA, they can also be applied to other low-rank
approximation techniques.

In contrast to constructing the approximation~$\tilde A$ of~$A$ the usual way and solving a single linear
system~$\tilde Ax=b$, we generate a sequence of $\Hm$-matrix approximations~$A_k$ to~$A$ and solve
each linear system~$A_kx_k=b$ for~$x_k$. At first glance this seems to be more costly than solving $\tilde Ax=b$ only once.
However, the construction of the next approximation~$A_{k+1}$ can be steered by the approximate solution~$x_k$, and $x_{k+1}$ can be
efficiently computed as an update of~$x_k$. Furthermore, the accuracy of~$x_k$ in the iterative solver can be adapted to the error estimator.
This allows to interlace the assembling process of the matrix with the iterative solution of the linear system. In addition to reducing 
the numerical effort for assembling and storing the matrix, this approach
has the practical advantage that the block-wise accuracy of ACA is not a required parameter any more. Choosing this parameter in usual ACA
approximations is not obvious as the relation between the block-wise accuracy and the error of the solution usually depends on the
respective problem. With the new approach presented in this article the accuracy of the approximation is automatically adapted.

The article is organised as follows. At first the model problem is formulated for which the new method will be explained.
After that we give a brief overview of hierarchical matrices and ACA as these methods are the foundations of our new method.
Suitable error estimators will be proposed and investigated with respect to efficiency and reliability in section~four. Using these, a marking and refinement strategy
will be presented and analysed. Numerical examples will be presented at the end of the article which show that the new method is more
efficient than ACA with respect to both computational time and required storage in situations where the solution in some sense contains structural differences or situations
that result from locally over-refined meshes.

\section{Model Problem}\label{sec:two}
As a model problem we consider the Laplace equation on a Lipschitz domain~$\Omega \subset \R^d$ (or its complement $\Omega^c:=\R^d\setminus\overline\Omega$), i.e.
\begin{equation} \label{mp}
 \begin{split}
  -\Delta u &= 0 \quad \text{in } \Omega \text{ (or }\Omega^c), \\
  u &= g \quad \text{on } \partial\Omega,
 \end{split}
\end{equation}
where the Dirichlet data $g$ is a given function in the trace space
$H^{1/2}(\partial \Omega)$ of the Sobolev space $H^1(\Omega)$. This kind of problem (in particular the exterior boundary value problem with suitable conditions at infinity) can be treated 
via boundary integral formulation using the fundamental solution
\begin{equation*}
 S(x) = \begin{cases} -\frac{1}{2\pi} \ln |x|, & d = 2, \\
  \frac{1}{(d-2)\omega_{d}} |x|^{2-d}, & d > 2,
  \end{cases}
\end{equation*}
for $x \in \R^d\ohne\{0\}$, where $\omega_d$ denotes the area of the $d$-dimensional unit sphere.
Defining the operators 
\begin{equation*}
 \begin{split}
  \op{V}\psi(x) &:= \int_{\partial \Omega} \psi(y)\, S(x-y) \ud s_y \quad \text{(single layer potential)},\\
  \op{K}\phi(x) &:= \int_{\partial \Omega} \phi(y) \,\partial_{n(y)}S(x-y) \ud s_y \quad \text{(double layer potential)},
 \end{split}
\end{equation*}
the solution of \eqref{mp} is given via the \textit{representation formula}
\[
u(x)=\op{V}\psi(x)-\op{K}g(x),\quad x\in\R^d\setminus\partial\Omega,
\]
where the Neumann data $\psi\in H^{-1/2}(\partial \Omega)$ is the solution of the boundary integral equation 
\begin{equation} \label{bp}
 \op{V}\psi = \left(\frac{1}{2}\op{I}+\op{K}\right)g \quad \text{on } \partial \Omega.
\end{equation}
If we extend the $L^2$-scalar product to a duality pairing between $H^{-1/2}(\partial \Omega)$ and $H^{1/2}(\partial \Omega)$, the boundary integral equation~\eqref{bp}
can be stated in variational form as
\begin{equation*}
 \scp{\op{V}\psi}{\psi'}_{L^2} = \scp{\left(  \frac{1}{2}\op{I}+\op{K} \right)g}{\psi'}_{L^2}, \quad \psi'
 \in H^{-1/2}(\partial \Omega).
\end{equation*}
Due to the coercivity of the bilinear form $\scp{\op{V}\cdot}{\cdot}_{L^2}$, the Riesz-Fischer theorem yields the existence of the unique solution $\psi \in H^{-1/2}(\partial \Omega)$.

In order to solve \eqref{bp} or its variational form numerically, Galerkin discretisations are often used. In the case $d=3$, the boundary $\partial \Omega$ is decomposed into a regular partition $\mathcal{T}$ consisting of $N$~triangles.
Let the set $\{\psi_{1},\ldots,\psi_{N}\}$ denote a basis of the the space of piecewise constant functions~$\mathcal{P}_{0}(\mathcal{T})$.
Then for $\psi = \sum_{j = 1}^{N} x_{j}\psi_{j}$, $x_{j} \in\R$, the discretisation of the variational problem reads
\begin{equation*}
\begin{split}
\scp{\left(\frac{1}{2}\op{I}+\op{K}\right)g}{\psi_{i}}_{L^2} &= \scp{\op{V}\psi}{\psi_{i}}_{L^2}
= \int_{\partial \Omega} \left[ \int_{\partial \Omega} S(x-y) \psi(y) \ud s_y\right] \psi_{i}(x) \ud s_x \\
&= \sum_{j = 1}^N x_{j}\left[ \int_{\partial \Omega} \int_{\partial \Omega} S(x-y) \psi_{j}(y) \psi_{i}(x) \ud s_y \ud s_x\right] , \quad i=1,\dots,N,
\end{split}
\end{equation*}
which due to the Riesz-Fischer theorem is uniquely solvable, too.
With
\begin{equation*}
 a_{ij} := \int_{\partial \Omega} \int_{\partial \Omega} S(x-y) \psi_{j}(y) \psi_{i}(x) \ud s_y \ud s_x \quad\text{and}\quad
  b_{i} := \scp{\left(  \frac{1}{2}\op{I}+\op{K} \right)g}{\psi_{i}}_{L^2}, \quad i,j = 1,...,N,
\end{equation*}
the above stated discretisation results in the system of linear equations
\begin{equation}\label{eq:Axb}
 A x = b, \quad A \in \R^{N \times N}, \; b \in \R^{N}.
\end{equation}
Since  the matrix $A$ is fully populated, our goal is to reduce the required storage for $A$ and the numerical effort for its construction to logarithmic-linear complexity.
To this end, hierarchical matrices  will be employed; see Sect.~\ref{sec:hmat}.
The efficiency of hierarchical matrices is based on a blockwise low-rank approximation.
The construction of these low-rank approximations is often done via adaptive cross approximation; see Sect.~\ref{sec:aca}.
While in hierarchical matrix approximations all blocks are commonly approximated with the same prescribed
accuracy, in this article we propose to construct the hierarchical matrix approximation in a block-adaptive
way, i.e., in addition to the adaptivity of ACA on each block we add another level of adaptivity to the
construction of the hierarchical matrix.

Although we consider the model problem~\eqref{mp}, the method presented in this article can be applied in more general situations as long as the matrix $A$ results from finite element discretisation of an integral representation of the problem.
Examples are the Helmholtz equation, Lam\'e equations, the Lippmann-Schwinger equation, fractional diffusion problems, the Gauss transform and many others.

\section{Hierarchical Matrices and the Adaptive Cross Approximation} \label{sec:hmat}
This section gives a short overview of the \textit{hierarchical matrix}~($\Hm$-matrix) structure and the ACA method, which we are going to extend. $\Hm$-matrices
were introduced by Hackbusch~\cite{hackbusch99} and Hackbusch and Khoromskij~\cite{hk00}. Their efficiency is based on a suitable hierarchical partitioning of the matrix into sub-blocks
and the blockwise approximation by low-rank matrices.

We consider matrices $A \in \R^{N \times N}$ whose entries 
result from the discretisation of a non-local operator such as the boundary integral formulation of the boundary value problem described in Sect.~\ref{sec:two}.
Although $A$ is fully populated, it can often be approximated by low-rank matrices on suitable blocks~$t\times s$ consisting of rows and columns~$t,s \subset I:= \{1,\dots,N\}$, of $A$, i.e.
\begin{equation*}
 A_{ts} \approx UV^{T}, \quad U \in \R^{t \times r}, \; V \in \R^{s \times r},
\end{equation*}
where the rank $r$ is small compared with $|t|$ and $|s|$. If $A$ discretises
an integral representation or the inverse of second-order elliptic partial differential
operators, then the condition 
\begin{equation}\label{eq:adm}
 \min\{ \diam\, X_{t}, \diam\, X_{s}\} < \beta \, \dist(X_{t},X_{s}), \quad \beta > 0,
\end{equation}
on the block $t\times s$ turns out to be appropriate; see \cite{bebendorf05}. Here, $X_t$ is the union of the
supports $X_{i} := \suppa\psi_{i}$, $i \in t$, and
\begin{equation*}
 \diam\, X = \sup_{x,y \in X} |x-y| \quad \text{and} \quad \dist(X,Y) = \inf_{x \in X,\, y \in Y} |x-y|
\end{equation*}
denote the diameter and the distance of two bounded sets $X, Y \subset \R^{d}$.
All blocks $t \times s$ satisfying this condition are called \textit{admissible}; see Fig.~\ref{fig:clusCoil}. 

\begin{figure}\centering
\includegraphics[angle=90,scale=0.25]{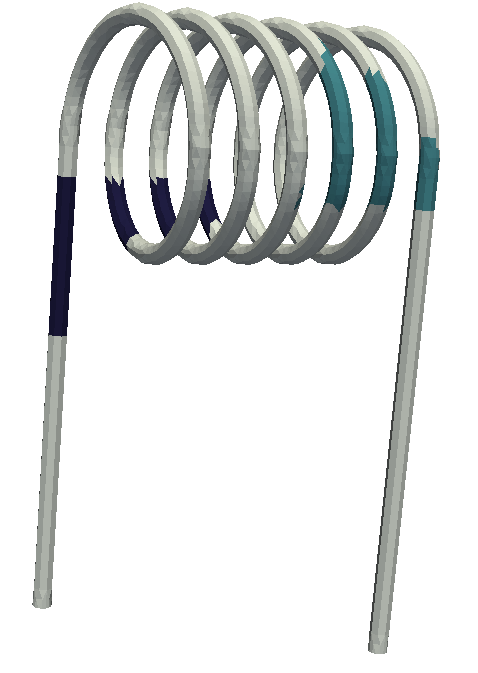}
 \caption{Two clusters $X_{t}$ and $X_{s}$ corresponding to an admissible block $t \times s$.}\label{fig:clusCoil}
\end{figure}
A partition $P$ of the matrix indices $I\times I$ is usually constructed using \textit{cluster trees}, i.e.\ trees $T_I$ satisfying the conditions
\begin{itemize}
 \item[(1)] $I$ is the root of the tree,
 \item[(2)] if $v \in T_I$ is not a leaf, then $v$ has sons $v_{1}, \dots, v_{s} \in T_I$ so that $v = v_{1} \dot\cup\ldots\dot\cup v_{s}$. 
\end{itemize}
The $\ell$-th level of $T_I$ will be referred to as $T_I^{(\ell)}$, $\ell=0,\dots,L$, where $L$ is the depth of $T_I$.
The set of leaves of the tree $T_I$ is denoted by $\mathcal{L}(T_{I})$.
Given a cluster tree~$T_I$ for $I$, a block cluster tree $T_{I\times I}$ for $I\times I$ is constructed by
recursive subdivision descending the cluster tree $T_I$ for the rows and for
the columns concurrently until \eqref{eq:adm} is satisfied or one the clusters cannot be subdivided any further.
Hence, a matrix partition consisting of blocks either satisfying~\eqref{eq:adm} or being small can be found as the leaves~$P:=\op{L}(T_{I\times I})$ of $T_{I\times I}$.
$P$ is the union of the sets of admissible and non-admissible blocks $P_\textnormal{adm}$ and $P_\textnormal{non-adm}$, respectively.
The \textit{sparsity constant} $c_{\textnormal{sp}}$ is defined as
\begin{equation*}
 c_{\textnormal{sp}} := \max\left\{ \max\limits_{t \in T_{I}} c_{\textnormal{sp}}^{r}(t), 
 \max\limits_{s \in T_{I}} c_{\textnormal{sp}}^{c}(s)\right\},
\end{equation*}
where \vspace*{-0.4cm}
\begin{equation*}
 \begin{split}
  c_{\textnormal{sp}}^{r}(t) &:= |\{ s \subset I : t \times s \in  P\}| 
   \end{split}
\end{equation*}
denotes the maximum number of blocks $t \times s$ contained in $P$ for a given cluster $t \in T_{I}$ and
 \begin{align*}
  c_{\textnormal{sp}}^{c}(s) &:=|\{ t \subset I  : t \times s \in P\}|
 \end{align*}
the maximum number of blocks $t \times s \in P$ for a cluster $s \in T_{I}$.
For more details on cluster trees (in particular for their construction) the reader is referred to~\cite{bebendorf08}.
 
Given a matrix partition $P$, we define the set of $\Hm$-matrices with blockwise rank at most~$r$ by
\begin{equation*}
 \Hm(P, r) := \left\{ M \in \R^{I \times I} : \rank\,  M_b \leq r \text{ for all } b \in P_\textnormal{adm} \right\}.
\end{equation*}
Due to the structure of hierarchical matrices and due to the fact that each block $t \times s$ has complexity $r(|t| + |s|)$, the complexity of the whole $\Hm$-matrix can be shown to be of the order $r\,N\log N$; see~\cite{bebendorf08}.

\subsection{Matrix construction}\label{sec:aca}
There are several ways how low-rank approximations of a matrix block can be constructed. For matrices~$A$ resulting from the discretisation of non-local operators the \textit{adaptive cross approximation}~(ACA) method introduced in \cite{bebendorf00, br03, bg06} constructs the low-rank approximation
using only few of the original entries of~$A$. Non-admissible blocks cannot be approximated but are small, so they are
computed entry by entry. 
We focus on a single admissible block $A_{ts} \in \R^{t \times s}$ of $A$. With $R_{0} := A_{ts}$ we define two sequences of vectors $u_r\in\R^t$ and $v_r\in\R^s$ for $r=1,2,3,\dots$
by the following algorithm.
\begin{algorithm}
\caption{Adaptive Cross Approximation (ACA)}\label{alg:31}
\begin{algorithmic}[0]
\State Let $r = 1$; $Z = \emptyset$; $\eps_\textnormal{ACA}> 0$
\Repeat
\State find $i_{r}$ by some rule
\State $\tilde{v}_{r} := A_{i_{r}, s}$
\For {$l = 1,...,r-1$} $\tilde{v}_{r} := \tilde{v}_{r} - (u_{l})_{i_{r}}v_{l}$ \EndFor
\State $Z := Z \cup \{i_{r}\}$
\If{$\tilde{v}_{r}$ does not vanish} 
\State $j_{r} := \textnormal{argmax}_{j\in s} |(\tilde{v}_{r})_{j}|$; $v_{r} := (\tilde{v}_{r})_{j_{r}}^{-1}\tilde{v}_{r}$ 
\State $u_{r} := A_{t, j_{r}}$
\For {$l = 1,\ldots, r-1$} 
$u_{r} := u_{r} - (v_{l})_{j_{r}} u_{l}$\EndFor
\State $r := r + 1$
\EndIf
\Until $\norm{u_{r+1}}_2\norm{v_{r+1}}_2 \leq \frac{\eps_\textnormal{ACA}(1-\beta)}{1+\eps_\textnormal{ACA}} \norm{S_r}_F$ or $Z = t$
\end{algorithmic}
\end{algorithm}\\
The matrix \begin{equation}\label{eq:recRep}
  S_r := \sum_{l = 1}^{r} u_{l}v_{l}^{T}
\end{equation}
has rank at most $r$ and can be shown (under reasonable assumptions) to be
an approximation of $A_{ts}$ with remainder $R_r:=A_{ts}- S_r$ having relative accuracy
\begin{equation}\label{eq:relaccR}
\norm{R_r}_F\leq\eps_\textnormal{ACA}\norm{A_{ts}}_F,
\end{equation}
where $\norm{\cdot}_F$ denotes the Frobenius norm.
The set $Z$ collects all vanishing rows of the matrix $R_{r}$.
The vectors $u_r$ and $\tilde v_r$ are columns and rows of $R_{r-1}$, i.e.
\[
  u_r=(R_{r-1})_{t j_r},\quad \tilde v_r=(R_{r-1})_{i_r s}.
  \]
  The row indices $i_r$ have to be chosen such that the Vandermonde matrix corresponding to the system in which the
  approximation error is to be estimated is non-singular; for details see \cite{bebendorf08}.
If the $i_{r}$-th row of $R_{r-1}$ is nonzero and hence is
used as $v_{r}$ after scaling $\tilde v_r$ with $1/(\tilde v_r)_{j_r}$, it is also added to $Z$ since the $i_{r}$-th row of the following remainder $R_{r}$ will vanish. It can be shown that the numerical effort of
ACA is of the order $|Z|^{2}(|t|+|s|)$.

\section{Error Estimators and the Block-Adaptive Cross Approximation (BACA)}\label{sec:convana}
The usual way of treating discrete problems~\eqref{eq:Axb} is to construct an hierarchical matrix approximation~$\tilde A$ (e.g.\ via ACA) such that each block $\tilde A_b$, $b\in P$, satisfies (cf.~\eqref{eq:relaccR})
\begin{equation}\label{eq:Hrelappr}
  \norm{A_b-\tilde A_b}_F\leq \eps\norm{A_b}_F \quad \text{for all }b\in P.
  \end{equation}
Due to the independence of the blocks, this can be done in parallel; see \cite{BeKr04}.
While \eqref{eq:Hrelappr} guarantees local properties of the $\Hm$-matrix approximation,
the global impact of this condition is difficult to estimate.
Notice that \eqref{eq:Hrelappr} still implies the obvious global property
\[
  \norm{A-\tilde A}_F\leq \eps\norm{A}_F,
\]
however, for instance the eigenpairs of $A$ and $\tilde A$ are not as accurate in general and one has to apply suitable techniques (see \cite{MBMBMB13}) in order to be able to guarantee spectral equivalence of $A$ and $\tilde A$. 
In other words, the low-rank approximation is usually constructed
regardless of the importance of the respective block for global properties of the matrix. 

$\tilde A$ obtained from usual ACA may not be the ideal approximation to $A$ when the
norm of the solution error of the associated solution is the measure. 
In order to find an $\Hm$-matrix approximation that is better suited for this problem, we employ techniques known from adaptivity together with suitable error estimators. The strategy of error estimation is well known in the context of numerical methods for partial differential or integral equations. 
 There are several types of a posteriori error estimators. The method of this article is inspired by the $(h-h/2)$-version investigated in~\cite{fp08}.
 Adaptive methods usually focus on mesh refinement. 
Here, these strategies will be used to successively improve
block-wise low-rank approximations, whereas the underlying grid (and the hierarchical block structure~$P$) does not change.

In contrast to constructing the matrix approximation $\tilde A$ of $A$ in the usual way (via ACA) and solving a single linear system $\tilde Ax=b$, we build a sequence of $\Hm$-matrix approximations~$A_k$ of~$A$ and solve
each linear system $A_kx_k=b$ for $x_k$. At first glance this seems to be more costly than solving $\tilde Ax=b$ only once. However, the approximation of $A_{k+1}$
can be steered by the approximate solution~$x_k$, and $x_{k+1}$ can be
computed as an update of $x_k$. Furthermore, the accuracy of $x_k$ can be low for small~$k$ and we can adapt it to the estimated error.

The following method will be based on \textit{residual error estimators}. Notice that if $\op{A}:V\to V'$ denotes the
operator that is discretised by~$A$ and $u_h$ refers to the Galerkin approximation of the exact solution to $\op{A}u=f$, then
for the residual error $\norm{f-\op{A}u_h}_{V'}$ we have
\[
c_B\norm{u-u_h}_V\leq\norm{f-\op{A}u_h}_{V'}=\sup_{\fie\in V} \frac{|a(u-u_h,\fie)|}{\norm{\fie}_V}\leq c_S\norm{u-u_h}_V
\]
provided that the continuous bilinear form $a:V\times V\to\R$ corresponding to $\op{A}$, i.e.\ $|a(u,v)|\leq c_S\norm{u}_V\norm{v}_V$ for all $u,v\in V$, satisfies an \textit{inf-sup condition}
\[
  \inf_{u\in V}\sup_{v\in V} \frac{|a(u,v)|}{\norm{u}_V\norm{v}_V}\geq c_B.
  \]
 Hence, the solution error $\norm{u-u_h}_{V}$ is equivalent with the residual error $\norm{f-\op{A}u_h}_{V'}$ and in the rest of this section
 we may focus on the residual error.

In the following algorithm, $\hat A_k$ denotes a better approximation of $A$ than $A_k$,
i.e.\ we assume that the \textit{saturation condition}
\begin{equation}\label{eq:satass}
\norm{\hat A_k x_k-A x_k}\leq \csat\norm{A_k x_k-Ax_k}
\end{equation}
holds with $0<\csat<1$.
A natural choice for $\hat A_k$ is the improved approximation that results from~$A_k$ by
adding a fixed number of additional steps of ACA for each admissible block and by setting $(\hat A_k)_b=A_b$ for all other blocks $b\in P_\textnormal{non-adm}$.
We define the matrix \[
  W_k:=A_k-\hat A_k.
  \]
Then $(W_k)_b=0$ for all $b\in P_\textnormal{non-adm}$.
\begin{algorithm}
\caption{Block-adaptive ACA}\label{alg:41}
\begin{algorithmic}[0]
\State 
\begin{enumerate}
 \item Start with a coarse $\Hm$-matrix approximation $A_{0}$ of $A$ and set $k=0$.
 \item Solve the linear system $A_k x_k=b$ for $x_k$ with residual error $\delta_k$ (use $x_{k-1}$ as a starting vector of an iterative solver; $x_{-1}:=0$).
 \item Given $0<\theta<1$, find a set of marked blocks $M_k\subset P_\textnormal{adm}$ with minimal cardinality such that
 \begin{equation}\label{Dmark}
 \eta_k(M_k)\geq \theta\,\eta_k,
 \end{equation}
 where $\eta_k^2(M):=\sum_{t\times s\in M} \norm{(W_k)_{ts}(x_k)_s}_2^2$ and
 $\eta_k:=\eta_k(P_\textnormal{adm})$.
 \item Let
 \[
   A_{k+1}=\begin{cases} (\hat A_k)_b, & b\in M_k,\\
          (A_k)_b, & b\in P\setminus M_k.
   \end{cases}
   \]
  \item If $\eta_k>\eps_\textnormal{BACA}$ increment $k$ and go to 2.
\end{enumerate}
\end{algorithmic}
\end{algorithm}

Notice that Algorithm~\ref{alg:31} contains
a stopping parameter~$\eps_\textnormal{ACA}$, which describes the desired accuracy of the low-rank approximation for the respective block.
While it is difficult to choose $\eps_\textnormal{ACA}$ so that the resulting solution of the linear system satisfies a prescribed accuracy,
the parameter $\eps_\textnormal{BACA}$ in Algorithm~\ref{alg:41} gives an upper bound on the residual error $\norm{b-Ax_k}_2$ of~$x_k$ as we shall see in the next Lemma~\ref{lem:rel}.

It seems that two $\Hm$-matrices $A_k$ and $\hat A_k$ are required for the previous algorithm.
Since they are strongly related, it is actually sufficient to store only $\hat A_k$. 
Notice that the ACA steps required for the construction of $A_{k+1}$ can be adopted from $\hat A_k$, and $\hat A_{k+1}$ can be
computed from $\hat A_k$ by improving the blocks contained in $M_k$.

In addition to the locatability of the error estimator 
\begin{equation} \label{eq:errEst}
  \eta_k^2=\sum_{t\times s\in P_\textnormal{adm}} \norm{(W_k)_{ts}(x_k)_s}_2^2
\end{equation}
introduced in Algorithm~\ref{alg:41}
its reliability is a main issue for a successful adaptive procedure. In particular, the dependence
on the residual error $\delta_k:=\norm{b-A_kx_k}_2$ of the solution $x_k$ of $A_kx_k=b$ in Algorithm~\ref{alg:41} has to be investigated.

\begin{lemma}\label{lem:rel}
Let assumption \eqref{eq:satass} be valid and let $\delta_k\leq \alpha\norm{W_kx_k}_2$ with some constant $\alpha\geq 0$.
Then $\eta_k$ is reliable, i.e., it holds
\[
\norm{b-Ax_k}_2\leq \frac{1+\alpha(1+\csat)}{1-\csat}\,\norm{W_kx_k}_2\leq \sqrt{\csp L}\,\frac{1+\alpha(1+\csat)}{1-\csat}\,\eta_k.
\]
\end{lemma}
\begin{proof}
We make use of the decomposition of $A\in\Hm(P,k)$ into a sum of
level matrices~$A^{(\ell)}$, which contain the blocks $b\in P$ of $A$ belonging to the $\ell$-th level of the block cluster tree $T_{I\times I}$
\[
A=\sum_{\ell=1}^L A^{(\ell)}.
\]
Notice that $A^{(\ell)}$ is a Cartesian product block matrix.
Due to
$(\sum_{i=1}^n a_i)^2\leq n\sum_{i=1}^n a_i^2$ for all $a_i\in\R$, $i=1,\dots,n$,
we have
\begin{align*}
\norm{W_k x_k}^2_2&\leq \left(\sum_{\ell=1}^L \norm{W_k^{(\ell)}x_k}_2\right)^2\leq L\sum_{\ell=1}^L\norm{W^{(\ell)}_k x_k}^2_2
=L\sum_{\ell=1}^L \sum_{t\in T_I^{(\ell)}} \norm{\sum_{s:t\times s\in P} (W_k)_{ts}(x_k)_s}^2_2\\
&\leq L\sum_{\ell=1}^L \sum_{t\in T_I^{(\ell)}} \left(\sum_{s:t\times s\in P}\norm{ (W_k)_{ts}(x_k)_s}_2\right)^2
\leq \csp L\sum_{\ell=1}^L \sum_{t\in T_I^{(\ell)}} \sum_{s:t\times s\in P}\norm{(W_k)_{ts}(x_k)_s}_2^2\\
&=\csp L\sum_{t\times s\in P}\norm{(W_k)_{ts}(x_k)_s}_2^2= \csp L\,\eta_k^2.
\end{align*}
The lemma follows from
\begin{align*}
\norm{b-Ax_k}_2&\leq \norm{b-A_kx_k}_2+\norm{W_k x_k}_2+\norm{\hat A_kx_k-A x_k}_2
\leq (1+\alpha)\norm{W_kx_k}_2+\csat\norm{A_kx_k-A x_k}_2\\
&\leq (1+\alpha)\norm{W_kx_k}_2+\csat(\alpha \norm{W_kx_k}_2+ \norm{b-A x_k}_2).
\end{align*}
\end{proof}

Although the efficiency of the error estimator~$\eta_k$, i.e.\ $\eta_k\lesssim\norm{b-Ax_k}_2$, can be observed in numerical experiments, its provability is still
an open question. Nevertheless, the computable expression $\norm{W_kx_k}_2$ can be used as a lower bound for $\norm{b-Ax_k}_2$.
\begin{lemma} \label{lem:eff}
  Let \eqref{eq:satass} be valid and let $\delta_k\leq \alpha\norm{W_kx_k}_2$ with $0\leq \alpha\leq 1/2$.
  Then \[\norm{b-Ax_k}_2\geq \frac{1-\alpha(1+\csat)}{1+\csat}\norm{W_kx_k}_2.\]
\end{lemma}
\begin{proof}
The saturation assumption \eqref{eq:satass} yields
\begin{align*}
\norm{W_kx_k}_2&\leq \norm{A_kx_k-Ax_k}_2+\norm{Ax_k-\hat A_kx_k}_2
\leq (1+\csat)\norm{Ax_k-A_k x_k}_2\\
&\leq (1+\csat)(\norm{b-Ax_k}_2+\norm{b-A_kx_k}_2)\leq (1+\csat)\norm{b-Ax_k}_2+(1+\csat)\alpha\norm{W_kx_k}_2.
\end{align*}
\end{proof}

In the following lemma the convergence of the error estimator $\eta_k$ is investigated.
Its convergence is a consequence of the \textit{D\"orfler marking strategy}~\eqref{Dmark}. To prove it, we use the
obvious property
\begin{equation}\label{eq:diffAh}
  \lim_{k\to\infty} \norm{A_{k+1}-A_k}=0=\lim_{k\to\infty} \norm{\hat A_{k+1}-\hat A_k}.
\end{equation}
\begin{lemma}\label{lem:opred}
Assume that the sequence $\{\norm{A_k^{-1}}_2\}_{k\in\N_0}$ is bounded and that $\delta_k \rightarrow 0$ for $k \rightarrow \infty$. Then it holds that
\begin{equation}\label{eq:errred}
  \eta^2_{k+1}\leq q\,\eta^2_k+z_k,
\end{equation}
where $z_k$ converges to zero and $q:=1-\frac{1}{2}\theta^2<1$. Furthermore, $\lim_{k\to\infty}\eta_k=0$.
\end{lemma}
\begin{proof} 
Using the definition of $A_{k+1}$, $\eta_{k+1}^2$ reads
\[
\sum_{t\times s\in P} \norm{(W_{k+1})_{ts} (x_{k+1})_s}_2^2
=\sum_{t\times s\in M_k} \norm{(\hat A_k-\hat A_{k+1})_{ts} (x_{k+1})_s}_2^2
+\sum_{t\times s\not\in M_k} \norm{(A_k-\hat A_k)_{ts} (x_{k+1})_s}_2^2.
\]
Due to Young's inequality and \eqref{Dmark} we have for the second term
\begin{align*}
  \sum_{t\times s\not\in M_k} &\norm{(W_k)_{ts} (x_{k+1})_s}_2^2
  \leq (1+\epsilon)\,\eta_k^2(P\setminus M_k)+
  (1+1/\epsilon)\sum_{t\times s\not\in M_k} \norm{(W_k)_{ts} (x_{k+1}-x_k)_s}_2^2\\
  &= (1+\epsilon)[\eta_k^2-\eta_k^2(M_k)]+(1+1/\epsilon)\sum_{t\times s\not\in M_k} \norm{(W_k)_{ts} (x_{k+1}-x_k)_s}_2^2\\
  &\leq (1+\epsilon)(1-\theta^2)\,\eta_k^2+(1+1/\epsilon)\sum_{t\times s\not\in M_k} \norm{(W_k)_{ts} (x_{k+1}-x_k)_s}_2^2
\end{align*}
for all $\epsilon>0$.
The choice $\epsilon:=\frac{1}{2}\theta^2/(1-\theta^2)$ and
\begin{align*}
z_k&:=\sum_{t\times s\in M_k} \norm{(\hat A_k-\hat A_{k+1})_{ts} (x_{k+1})_s}_2^2+ (1+1/\epsilon)\sum_{t\times s\not\in M_k} \norm{(W_k)_{ts} (x_{k+1}-x_k)_s}_2^2\\
&\leq \csp L \max_{t\times s\in P}\norm{(\hat A_k-\hat A_{k+1})_{ts}}_2^2\,\norm{x_{k+1}}_2^2 \\
&\quad +(1+1/\epsilon)\csp L\max_{t\times s\in P}\norm{(W_k)_{ts}}_2^2\,\norm{A_{k+1}^{-1}}^2_2\,\norm{A_{k+1}x_{k+1}-A_{k+1}x_k}_2^2\\
&\leq \csp L \Big[\norm{\hat A_k - \hat A_{k+1}}_2^2 \, \norm{x_{k+1}}_2^2 \\
&\quad +(1+1/\epsilon) \norm{W_k}_2^2\,\norm{A_{k+1}^{-1}}^2_2 \, \Big(\norm{A_{k+1}x_{k+1}-b}_2^2 + \norm{(A_{k+1}-A_k)x_k}_2^2 + \norm{A_kx_k-b}_2^2 \Big)\Big]\\
&\leq \csp L \Big[\norm{\hat A_k - \hat A_{k+1}}_2^2 \, \norm{x_{k+1}}_2^2 \\
&\quad +(1+1/\epsilon) \norm{W_k}_2^2\,\norm{A_{k+1}^{-1}}^2_2 \,\Big(\delta_{k+1}^2 + \norm{A_{k+1}-A_k}_2^2\, \norm{x_k}_2^2 + \delta_k^2 \Big)\Big]
\end{align*}
leads to the first part of the assertion due to the boundedness of $\norm{A_k^{-1}}_2$, the convergence of $\delta_k$ to~$0$ and~\eqref{eq:diffAh}.

For the proof of the second part we pursue the ideas of the estimator reduction principle, which was introduced in \cite{afl12}.
Let $z > 0$ be a number with $z_{k} \leq 
z$ for all $k$.  With the estimator reduction~\eqref{eq:errred} it follows that
\begin{equation*}
 \begin{split}
  \eta^2_{k+1} &\leq q\, \eta^2_{k} + z_k 
   \leq q^2\, \eta_{k-1}^2 + q\, z_{k-1}+z_k\leq\ldots
  \leq q^{k+1} \eta_{0}^2 + \sum_{i = 0}^{k} q^{k-i} z_i \\
   &\leq q^{k+1} \eta_{0}^2 + z \sum_{l = 0}^{k} q^l
 \leq \eta_{0}^2 + \frac{z}{1 - q}.
 \end{split}
\end{equation*}
Hence, the sequence $\{\eta_{k}\}_{k \in \mathbb{N}_{0}}$ is bounded and we define $M:=\limsup_{k \rightarrow \infty} \eta_{k}^2$.
Using the estimator reduction~\eqref{eq:errred} once more, leads to
\begin{equation*}
 M = \limsup_{k \rightarrow \infty} \eta_{k+1}^2 \leq q \limsup_{k \rightarrow \infty} \eta^2_{k} + \limsup_{k \rightarrow \infty} 
 z_{k} = qM.
\end{equation*}
Hence, $M=0$ and we obtain
\begin{equation*}
 0 \leq \liminf_{k \rightarrow \infty} \eta_{k} \leq \limsup_{k \rightarrow \infty} \eta_{k} = 0
\end{equation*}
and finally $\lim_{k \rightarrow \infty} \eta_{k} = 0$.
\end{proof}

\begin{theorem}
  The residuals $r_k:=b-Ax_k$ of the sequence $\{x_k\}_{k\in\N}$ constructed by Algorithm~\ref{alg:41} converge to zero.
\end{theorem}
\begin{proof}
With Lemma~\ref{lem:rel} and Lemma~\ref{lem:opred} we obtain that
\[
  \norm{b-Ax_k}_2\leq \sqrt{\csp  L}\,\frac{1+\alpha(1+\csat)}{1-\csat}\,\eta_k\to0.
  \]
\end{proof}

  \section{Numerical Experiments}\label{sec:5}
  The numerical experiments are subdivided into three parts. At first we investigate the storage requirements of BACA. After that the quality of the proposed error
  estimator is discussed. Finally, the third part treats the acceleration of the matrix approximation and solution of the boundary value problem in comparison with ACA.
  In order to be able to compare these results with the results obtained
  from ACA, we prescribe the same error of $u_h$. Therefore, \[
    e_h := \frac{\norm{u - u_h}_{L^2}}{\norm{u}_{L^2}}
    \] describes the relative error of the solution $u_h$.
  
  The approximation steps are executed without any parallelisation in both cases. We use 
  the conjugate gradient method  as a solver for the arising linear systems of equations without preconditioning.
  In the case of BACA, the accuracy of the iterative solver is adapted to the size of~$\norm{W_kx_k}_2$; see Lemma~\ref{lem:rel}.
  For all the experiments that are based on ACA the accuracy of CG is set to~$10^{-8}$. 
  
  As numerical examples we consider a family of boundary value problems, in which the singularity of the right-hand side is moved towards the boundary of the computational domain~$\Omega$, i.e.
  \begin{subequations} \label{eq:numEx}
  \begin{alignat}{2}
    -\Delta u &= 0 &  &\text{ in } \Omega:=B_1(0), \\
   u(x) &= S(x - p_{i}) &\quad &\text{ on } \partial \Omega, \ i = 1,2,3,4,
   \end{alignat}
\end{subequations}
  for $p_{i} = (x_{i}, 0, 0)^{T}$ with $x_1 = 10.0$, $x_2 = 1.5$, $x_3 = 1.1$, $x_4 = 1.05$. In these tests we use the following
  parameters:
  the minimal block size $b_{\min} = 15$, the blockwise accuracy $\varepsilon_{\text{ACA}} = 10^{-6}$ of ACA, and the admissibility parameter
  $\beta = 0.8$.

  \subsection{Storage Reduction} \label{sec:51}
  We start with a uniform decomposition of the unit ball into 642 points and 1280 triangles for the boundary value problems~\eqref{eq:numEx}.
  Note that ACA does not take into account the right-hand side. Hence the approximation does not depend on the choice of $p_i$. So, we obtain the left picture of 
  Fig.~\ref{fig:singLay3} after the application of ACA to the discrete integral operator in each of the cases $i = 1,\dots,4$.

  The proposed algorithm BACA provides the results shown in Tab.~\ref{tab:NumRes2}, in which the parameter $\theta = 0.9$ was used.
  Here, the ranks of $\hat A_k$ are ahead of $A_k$ by two ACA steps. The parameter $\alpha$ from Lemma~\ref{lem:rel} is set to $100$.
   Compared with the approximation obtained from ACA, we obtain lower storage requirements for the matrix approximation~$\hat A_k$.
   
   \begin{table}[htb] \centering
   \begin{tabular}{r | c c c |c c c  | c  c  c}
      & 	 \multicolumn{6}{c}{BACA}	 & 	\multicolumn{3}{|c}{ACA}		 \\
      $i$   & blockwise & $\varepsilon_{\text{BACA}}$  &  $\norm{b - Ax_k}_2$&  $e_h$& storage & compr. & $e_h$ & storage & compr. \\
 & rank $A_0$&  & & & (MB) & (\%) & & (MB) & (\%) \\ \hline 
    1 & 6 & 1e-08 & 1.89e-08 & 0.002 & 3.43 & 54.9 & 0.002 & 3.85 & 61.5 \\[2pt]
    2 & 4 & 5e-06 & 6.08e-06 & 0.033 & 2.86 & 45.7 & 0.033 & 3.85 & 61.5 \\[2pt]
    3 & 3 & 1e-04 & 1.48e-04 & 0.264 & 2.40 & 38.4 & 0.264 & 3.85 & 61.5 \\[2pt]
    4 & 2 & 5e-04 & 6.40e-04 & 0.951 & 2.10 & 33.5 & 0.951 & 3.85 & 61.5 \\
   \end{tabular}\vspace*{0.2cm}\\
   \caption{Numerical results of BACA for four positions of the singularity.} \label{tab:NumRes2}
  \end{table}

  The storage reduction and the improved compression rate observed in Tab.~\ref{tab:NumRes2} are also visible from Fig.~\ref{fig:singLay3},
  where the respective approximation $\hat A_k$ of $A$ with its blockwise ranks is shown (green blocks). Therein, red blocks are  
  constructed entry-wise without approximation.
  \begin{figure}[htb] \centering
    \scalebox{0.35}{\includegraphics{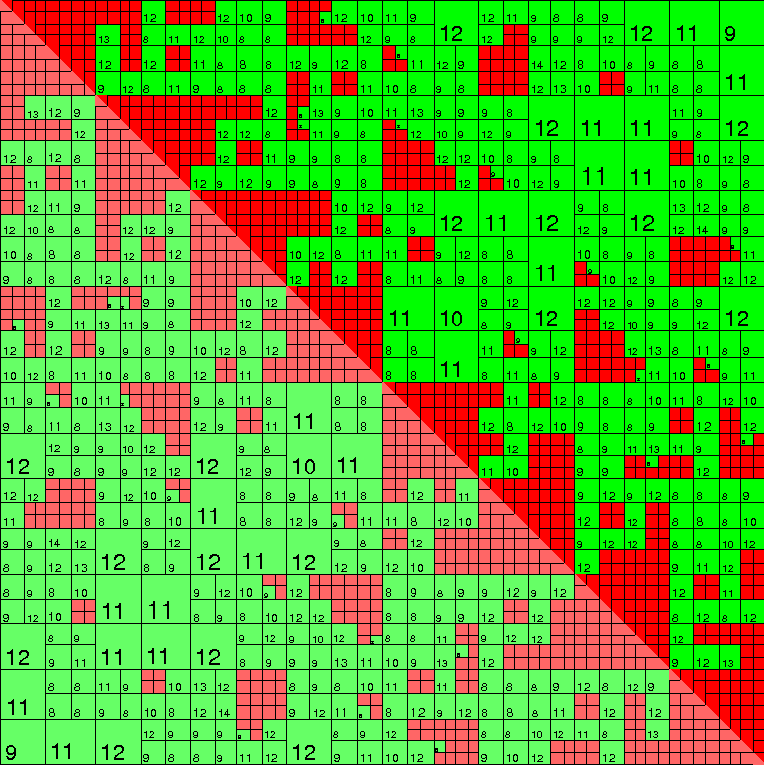}} \hspace*{1cm}
    \scalebox{0.35}{\includegraphics{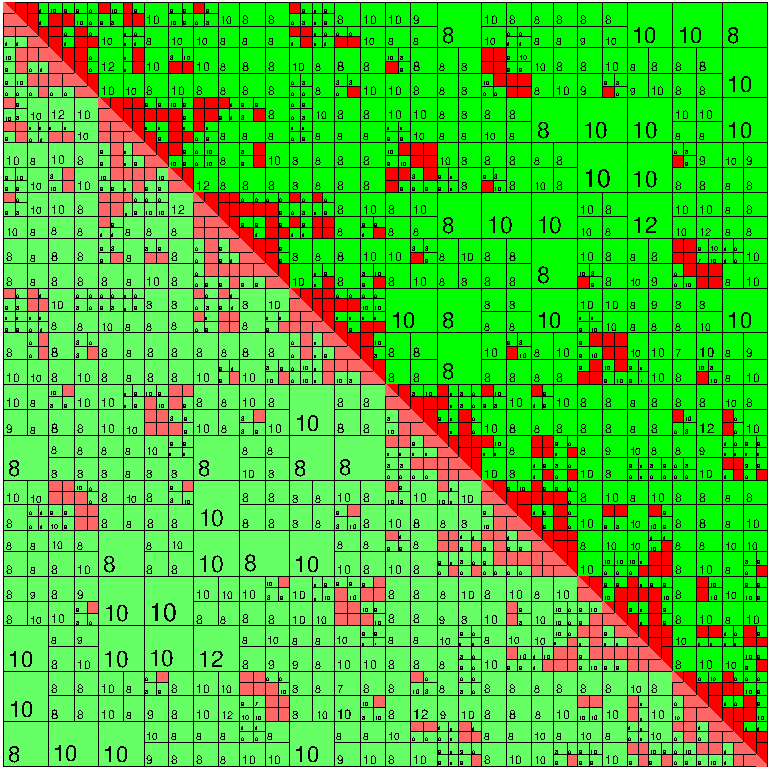}}\\
    \caption{Approximation of the blocks via ACA (left) and via BACA in the case $p_1$ (right).} \label{fig:singLay3}
  \end{figure}

  Example~\eqref{eq:numEx} has been chosen in order to investigate different right-hand sides which lead to structural differences in the respective solution.
  Notice that the influence of the singularity on the solution is stronger the smaller the distance of $p_i$ to the boundary of $\Omega$ becomes.
  Hence, for $p_i$ close to the boundary certain parts of the discrete integral operator are more important for the accuracy of the solution than others.
  Even for a relatively large distance, i.e.\ $p_1 = (10,0,0)^T$, smaller ranks and hence a better compression rate than ACA can be observed;
  see Fig.~\ref{fig:singLay3}. This effect is even stronger if $p_i$ approaches the boundary.
  The average ranks are shown in Tab.~\ref{tab:avRanks}. 
  \begin{table}[htb] \centering
   \begin{tabular}{c | c | c | c | c | c}
 	 matrix	 & 	 ACA		&	\multicolumn{4}{c}{BACA} \\
	&			&		$p_1$			&	$p_2$		&	$p_3$	&	$p_4$	\\ \hline 
    $A_k$	&	12.54		&		7.44			&	5.00		&	3.71	&	3.08	\\[2pt]
    $\hat{A}_k$	&	---		&		8.52			&	6.62		&	5.38	&	5.01
   \end{tabular}\vspace*{0.2cm}\\
   \caption{Average ranks.} \label{tab:avRanks}
  \end{table}
  So Tab.~\ref{tab:NumRes2} and Fig.~\ref{fig:singLay3} indicate that the proposed
  error estimator detects those matrix blocks
  which are important for the respective problem; see Fig.~\ref{fig:detect}.
  \begin{figure}[htb] \centering
    \scalebox{0.35}{\includegraphics{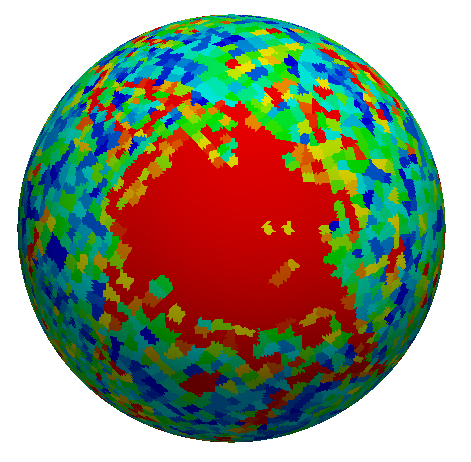}}
    \caption{Regions corresponding to large ranks (red).} \label{fig:detect}
  \end{figure}
  Therefore, BACA is able to construct an $\Hm$-matrix approximation
  that is particularly suited to solve a linear system for a specific right-hand side. 
  
  The improved storage requirements can be observed also for finer triangulations of the boundary. 
  Tab.~\ref{tab:NumN} shows the results for the boundary value problem \eqref{eq:numEx} in the case $i=3$ for several numbers
  of degrees of freedom~$N$. In addition to the compression rate and the storage requirements, the number of computed entries for the construction of~$\hat A_k$ are
  presented. Since the solution $u_h$ is expected to be more accurate the larger~$N$ is, the accuracy~$\eps_{\text{BACA}}$
  of the error estimator has to be adapted to $N$. All other parameters remain unchanged. In comparison, the results of ACA applied to the boundary value 
  problem \eqref{eq:numEx} in the case $i=3$ can be seen in Tab.~\ref{tab:NumN_ACA}.
 
     \begin{table}[htb] \centering
   \scalebox{1.0}{
  \begin{tabular}{ r |  c c | c  c r r}
    $N$ & $\varepsilon_{\text{BACA}}$  &  $\norm{b - Ax_k}_2$&  $e_h$ & number of computed & storage & compr.  \\
	& & & &entries& (MB) & (\%) \\ \hline
  1\,280	&1e-04	&4.14e-04	&0.264	&3.22e05	&2.48		&39.7		\\ [2pt]
  7\,168	&1e-05	&2.31e-05	&0.077	&3.32e06	&25.47		&13.0		\\ [2pt]
  28\,672	&1e-06	&2.56e-06	&0.023 	&2.10e07	&160.98		&5.1		\\ [2pt]
 114\,688 	&1e-06	&2.13e-06	&0.007 	&8.76e07	&837.92		&1.7		\\ 
  \end{tabular}}
  \caption{Numerical results of BACA for $p_3$ and several $N$.} \label{tab:NumN}
  \end{table}
  
    \begin{table}[htb] \centering
   \scalebox{1.0}{
  \begin{tabular}{ r |c c r r}
    $N$ &$e_h$ & number of computed &  storage & compr. \\
	& & entries& (MB) & (\%) \\ \hline
  1\,280	&0.264	&5.01e05	&3.85		&61.5		\\ [2pt]
  7\,168	&0.077	&4.91e06	&37.53		&19.8		\\ [2pt]
  28\,672	&0.023 	&2.66e07	&203.70 	&6.5		\\ [2pt]
 114\,688 	&0.007	&1.35e08	&1034.56 	&2.1		\\ 
  \end{tabular}}
  \caption{Numerical results of ACA for $p_3$ and several $N$.} \label{tab:NumN_ACA}
  \end{table}
  
  Hence, BACA requires significantly less original matrix entries than the usual construction via ACA of the $\Hm$-matrix approximation to obtain approximately the same
  relative error of~$u_h$. 

  \subsection{Quality of the error estimator}
  In the following tests we validate the  error estimator $\eta_k$ introduced in~\eqref{eq:errEst}. At first the 
  reliability statement of Lemma~\ref{lem:rel} together with the lower bound $\norm{W_kx_k}_2$ of Lemma~\ref{lem:eff} 
  is investigated.
  Again, the boundary element method combined with BACA is tested with the problem described in \eqref{eq:numEx} and a triangulation with 14338 points
  and 28672 triangles. The accuracy $\eps_{\text{BACA}}$ of the error estimator is set to~$10^{-7}$.
  We start BACA with a coarse approximation $A_0$ which has been generated by applying four ACA steps to each block.
  Tab.~\ref{tab:errEst} and in Fig.~\ref{fig:errEst} contain the results for $\alpha = 1/2$ and three refinement parameters~$\theta$.
  The ranks of $\hat A_k$ are ahead of $A_k$ by three ACA steps.
  
  \begin{figure}[htb] \centering \scalebox{0.8}{
  \vbox{
  \begin{tikzpicture}
  \begin{semilogyaxis}[
       xmin = 1, xmax = 8, ymax = 1e-04, xlabel = $k$, 
  legend style = {at={(1,1)}, anchor = north east},
  title={$\theta = 0.9$}]
  \addplot[blue, very thick] coordinates{
  (1,5*1.36e-05)
  (2,5*6.07e-06)
  (3,5*2.71e-06)
  (4,5*1.20e-06)
  (5,5*5.60e-07)
  (6,5*2.56e-07)
  (7,5*1.19e-07)
  (8,5*5.96e-08)
  };
  \addlegendentry{$5\eta_k$}
  \addplot[blue, dotted, very thick] coordinates{
  (0,1.73e-05)
  (1,1.07e-05)
  (2,5.68e-06)
  (3,3.16e-06)
  (4,1.59e-06)
  (5,6.68e-07)
  (6,3.46e-07)
  (7,1.91e-07)
  (8,1.04e-07)
  };
  \addlegendentry{$\norm{b-Ax_k}_2$}
    \addplot[blue, dashed, very thick] coordinates{
  (1,0.5*1.98e-05)
  (2,0.5*8.30e-06)
  (3,0.5*3.81e-06)
  (4,0.5*1.59e-06)
  (5,0.5*7.12e-07)
  (6,0.5*3.01e-07)
  (7,0.5*1.30e-07)
  (8,0.5*6.33e-08)
  };
  \addlegendentry{$\frac{1}{2}\norm{W_k x_k}_2$}
  \end{semilogyaxis}
  \end{tikzpicture}
  \begin{tikzpicture}
  \begin{semilogyaxis}[
       xmin = 1, xmax = 17, ymax = 1e-04, xlabel = $k$,
  legend style = {at={(1,1)}, anchor = north east}, title={$\theta = 0.7$}]
    \addplot[red, very thick] coordinates{
  (1,5*1.36e-05)
  (2,5*9.77e-06)
  (3,5*6.96e-06)
  (4,5*4.98e-06)
  (5,5*3.57e-06)
  (6,5*2.58e-06)
  (7,5*1.85e-06)
  (8,5*1.32e-06)
  (9,5*9.52e-07)
  (10,5*6.96e-07)
  (11,5*5.04e-07)
  (12,5*3.64e-07)
  (13,5*2.65e-07)
  (14,5*1.92e-07)
  (15,5*1.40e-07)
  (16,5*1.04e-07)
  (17,5*7.79e-08)
  };
  \addlegendentry{$5 \eta_k$}
  \addplot[red, dotted, very thick] coordinates{
  (0,1.73e-05)
  (1,1.41e-05)
  (2,1.19e-05)
  (3,9.45e-06)
  (4,7.05e-06)
  (5,5.55e-06)
  (6,4.22e-06)
  (7,3.39e-06)
  (8,2.42e-06)
  (9,1.85e-06)
  (10,1.55e-06)
  (11,1.03e-06)
  (12,6.84e-07)
  (13,5.19e-07)
  (14,4.11e-07)
  (15,3.19e-07)
  (16,2.38e-07)
  (17,1.95e-07)
  };
  \addlegendentry{$\norm{b-Ax_k}_2$}
    \addplot[red, dashed, very thick] coordinates{
  (1,0.5*1.98e-05)
  (2,0.5*1.29e-05)
  (3,0.5*9.33e-06)
  (4,0.5*7.06e-06)
  (5,0.5*4.88e-06)
  (6,0.5*3.58e-06)
  (7,0.5*2.46e-06)
  (8,0.5*1.76e-06)
  (9,0.5*1.26e-06)
  (10,0.5*9.00e-07)
  (11,0.5*6.30e-07)
  (12,0.5*4.44e-07)
  (13,0.5*3.13e-07)
  (14,0.5*2.14e-07)
  (15,0.5*1.52e-07)
  (16,0.5*1.13e-07)
  (17,0.5*8.40e-08)
  };
  \addlegendentry{$\frac{1}{2}\norm{W_k x_k}_2$}
  \end{semilogyaxis}
  \end{tikzpicture} 
  \begin{tikzpicture}
  \begin{semilogyaxis}[
       xmin = 1, xmax = 24, ymax = 1e-04, xlabel = $k$,
  legend style = {at={(1,1)}, anchor = north east},, title={$\theta = 0.6$}]
  \addplot[black!50!green, very thick] coordinates{
  (1,5*1.36e-05)
  (2,5*1.08e-05)
  (3,5*8.49e-06)
  (4,5*6.78e-06)
  (5,5*5.41e-06)
  (6,5*4.33e-06)
  (7,5*3.47e-06)
  (8,5*2.80e-06)
  (9,5*2.24e-06)
  (10,5*1.79e-06)
  (11,5*1.44e-06)
  (12,5*1.15e-06)
  (13,5*9.31e-07)
  (14,5*7.53e-07)
  (15,5*6.07e-07)
  (16,5*4.90e-07)
  (17,5*3.94e-07)
  (18,5*3.19e-07)
  (19,5*2.57e-07)
  (20,5*2.08e-07)
  (21,5*1.69e-07)
  (22,5*1.37e-07)
  (23,5*1.12e-07)
  (24,5*9.31e-08)
  };
  \addlegendentry{$5\eta_k$}
  \addplot[black!50!green, dotted, very thick] coordinates{
  (0,1.73e-05)
  (1,1.52e-05)
  (2,1.31e-05)
  (3,1.18e-05)
  (4,9.78e-06)
  (5,8.24e-06)
  (6,6.87e-06)
  (7,5.93e-06)
  (8,4.90e-06)
  (9,4.15e-06)
  (10,3.53e-06)
  (11,3.11e-06)
  (12,2.37e-06)
  (13,1.99e-06)
  (14,1.70e-06)
  (15,1.34e-06)
  (16,1.10e-06)
  (17,8.41e-07)
  (18,6.70e-07)
  (19,5.68e-07)
  (20,4.73e-07)
  (21,4.06e-07)
  (22,3.38e-07)
  (23,2.89e-07)
  (24,2.28e-07)
  };
  \addlegendentry{$\norm{b-Ax_k}_2$}
    \addplot[black!50!green, dashed, very thick] coordinates{
  (1,0.5*1.98e-05)
  (2,0.5*1.40e-05)
  (3,0.5*1.16e-05)
  (4,0.5*8.92e-06)
  (5,0.5*7.50e-06)
  (6,0.5*6.30e-06)
  (7,0.5*4.73e-06)
  (8,0.5*3.92e-06)
  (9,0.5*3.04e-06)
  (10,0.5*2.36e-06)
  (11,0.5*1.92e-06)
  (12,0.5*1.50e-06)
  (13,0.5*1.24e-06)
  (14,0.5*9.86e-07)
  (15,0.5*7.72e-07)
  (16,0.5*6.13e-07)
  (17,0.5*4.88e-07)
  (18,0.5*3.82e-07)
  (19,0.5*3.03e-07)
  (20,0.5*2.39e-07)
  (21,0.5*1.82e-07)
  (22,0.5*1.48e-07)
  (23,0.5*1.23e-07)
  (24,0.5*1.01e-07)
  };
  \addlegendentry{$\frac{1}{2}\norm{W_k x_k}_2$}
  \end{semilogyaxis}
  \end{tikzpicture}
  }
  }
  \caption{The error estimator, the residual, and the lower bound $\norm{W_kx_k}_2$ for several parameters $\theta$.} \label{fig:errEst}
  \end{figure}
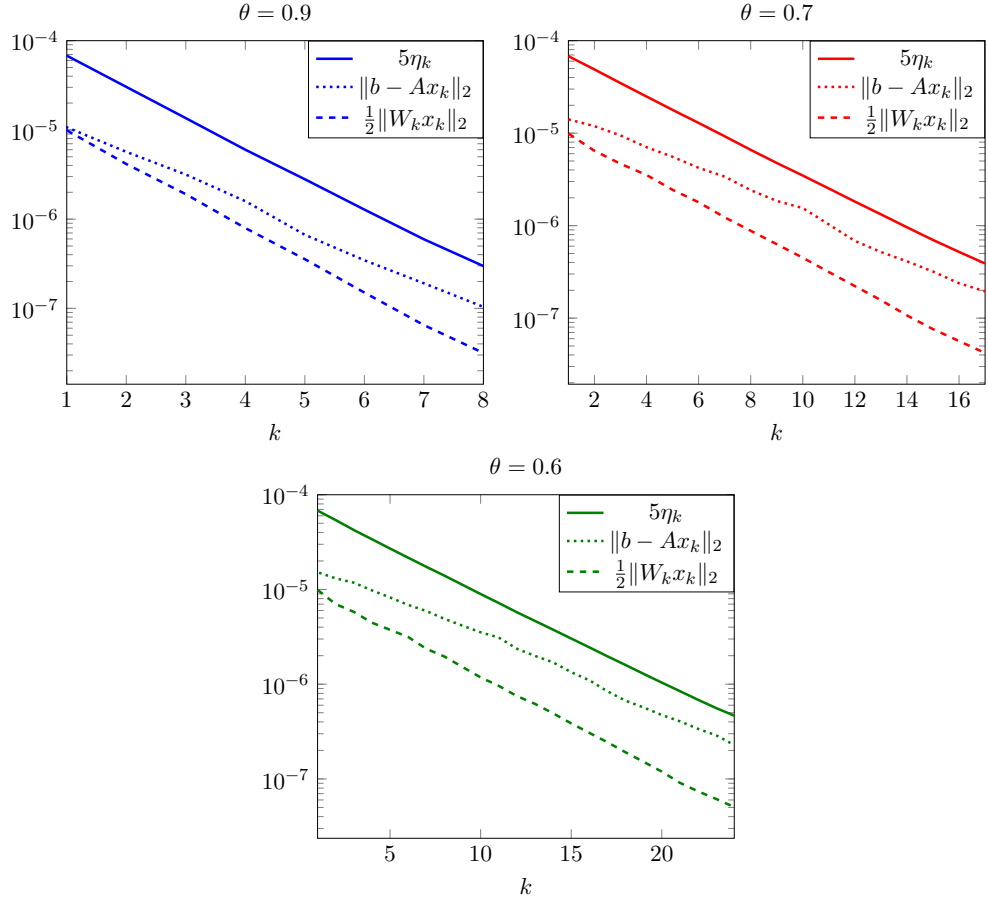
  The proposed error estimator~$\eta_k$ estimates the residual error $\norm{b-Ax_k}_2$ in an appropriate way.
  The expression $\frac{1}{2}\norm{W_kx_k}_2$ can serve as a lower bound for the residual error.
    Tab.~\ref{tab:errEst} shows that the convergence with respect to~$k$ is the faster the refinement parameter~$\theta$ is closer to~$1$.
  A larger parameter~$\theta$ in \eqref{Dmark} leads to a larger set~$M_k$ of marked blocks, which is likely to produce
  redundant information in~$A_{k+1}$. Conversely, small refinement parameters~$\theta$ usually lead to matrix approximations with lower storage requirements.
  
   \begin{table}[htb] \centering
  \begin{tabular}{  r | c c  c | c c c  }
  	& \multicolumn{3}{c|}{$\eta_{k}$}					&\multicolumn{3}{c}{$\norm{b-Ax_k}_2$}		\\	
  $k$ 	& $\theta = 0.6$	& $\theta = 0.7$	&$\theta = 0.9	$	& $\theta = 0.6$	& $\theta = 0.7$	&$\theta = 0.9$		\\ \hline 
  1	&1.36e-05		&1.36e-05		&1.36e-05		&1.52e-05		&1.41e-05		&1.07e-05			\\ 
  2	&1.08e-05		&9.77e-06		&6.07e-06		&1.31e-05		&1.19e-05		&5.68e-06			\\ 
  3	&8.49e-06		&6.96e-06		&2.71e-06		&1.18e-05		&9.45e-06		&3.16e-06			\\ 
  4	&6.78e-06		&4.98e-06		&1.20e-06		&9.78e-06		&7.05e-06		&1.59e-06			\\ 
  5	&5.41e-06		&3.57e-06		&5.60e-07		&8.24e-06		&5.55e-06		&6.68e-07			\\ 
  6	&4.33e-06		&2.58e-06		&2.56e-07		&6.87e-06		&4.22e-06		&3.46e-07			\\ 
  7	&3.47e-06		&1.85e-06		&1.19e-07		&5.93e-06		&3.39e-06		&1.91e-07			\\ 
  8	&2.80e-06		&1.32e-06		&5.96e-08		&4.90e-06		&2.42e-06		&1.04e-07			\\ 
  9	&2.24e-06		&9.52e-07		&---			&4.15e-06		&1.85e-06		&---				\\ 
  10	&1.79e-06		&6.96e-07		&---			&3.53e-06		&1.55e-06		&---				\\ 
  11	&1.44e-06		&5.04e-07		&---			&3.11e-06		&1.03e-06		&---				\\ 
  12	&1.15e-06		&3.64e-07		&---			&2.37e-06		&6.84e-07		&---				\\ 
  13	&9.31e-07		&2.65e-07		&---			&1.99e-06		&5.19e-07		&---				\\ 
  14	&7.53e-07		&1.92e-07		&---			&1.70e-06		&4.11e-07		&---				\\ 
  15	&6.07e-07		&1.40e-07		&---			&1.34e-06		&3.19e-07		&---				
  \end{tabular}
  \caption{Numerical results of BACA for different $\theta$.} \label{tab:errEst}
  \end{table}
  The numerical results confirm the theoretical findings of Lemma~\ref{lem:rel} and Lemma~\ref{lem:eff}.
  The proposed error estimator~$\eta_k$ is reliable and $\frac{1}{2}\norm{W_kx_k}_2$ can be used as a lower bound on the
  residual error. 
  
  \subsection{Acceleration of the matrix approximation} \label{sec:52}
  The results  represent only a part of our aims of the extension of ACA. As already mentioned in the introduction, we are also interested in an acceleration of the matrix 
  approximation and an acceleration of the solution process.
  Since the computation of matrix entries is by far the most time-consuming
  part of the boundary element method and we have seen that BACA requires less entries, we can expect improved computational times.

  We consider again the family of boundary value problems~\eqref{eq:numEx}. The ranks of $\hat A_k$ are ahead of $A_k$ by two ACA steps.
  and we choose $\alpha = 100$.
  The ratio of the time required for BACA and the time of the ACA can be observed in the Figs.~\ref{fig:time1} and~\ref{fig:time2} for two different triangulations.
  \begin{figure}[htb]
  \begin{minipage}{0.49\textwidth}
  \begin{tikzpicture}
  \draw[line width=0.7pt, ->, >=latex] (-0.5,0)--(5.5,0);
  \draw[line width=0.7pt, ->, >=latex] (0,-0.5)--(0,4.5);
  \draw[line width=0.7pt, fill = white!75!black] (0,0)--(1,0)--(1,4)--(0,4)--(0,0);
  \draw[line width=0.7pt, fill = blue!33!green] (1,0)--(2,0)--(2,3)--(1,3)--(1,0);
  \draw[line width=0.7pt, fill = blue!22!green] (2,0)--(3,0)--(3,2.64)--(2,2.64)--(2,0);
  \draw[line width=0.7pt, fill = blue!11!green] (3,0)--(4,0)--(4,2.2)--(3,2.2)--(3,0);
  \draw[line width=0.7pt, fill = blue!0!green] (4,0)--(5,0)--(5,1.92)--(4,1.92)--(4,0);
  \draw[line width=0.7pt] (0.1,2)--(-0.1,2) node[left]{0.5};
  \draw[line width=0.7pt] (0.1,4)--(-0.1,4) node[left]{1};
  \node[rectangle] at (-0.8, 4.8) {\small{$\frac{\textnormal{time(BACA)}}{\textnormal{time(ACA)}}$}};
  \node[rectangle] at (0.5, 4.2) {\small{1.00}};
  \node[rectangle] at (1.5, 3.2) {\small{0.75}};
  \node[rectangle] at (2.5, 2.84) {\small{0.66}};
  \node[rectangle] at (3.5, 2.4) {\small{0.55}};
  \node[rectangle] at (4.5, 2.12) {\small{0.48}};
  \node[rectangle] at (0.5, -0.3) {\small{ACA}};
  \node[rectangle] at (1.5, -0.3) {\small{$p_{1}$}};
  \node[rectangle] at (2.5, -0.3) {\small{$p_{2}$}};
  \node[rectangle] at (3.5, -0.3) {\small{$p_{3}$}};
  \node[rectangle] at (4.5, -0.3) {\small{$p_{4}$}};
  \node[rectangle] at (5.5, -0.3) {\small{$p_{i}$}};
  \end{tikzpicture}
  \caption{Triangulation with 7\,168 triangles} \label{fig:time1}
  \end{minipage}
  \begin{minipage}{0.49\textwidth}
  \begin{tikzpicture}
  \draw[line width=0.7pt, ->, >=latex] (-0.5,0)--(5.5,0);
  \draw[line width=0.7pt, ->, >=latex] (0,-0.5)--(0,4.5);
  \draw[line width=0.7pt, fill = white!75!black] (0,0)--(1,0)--(1,4)--(0,4)--(0,0);
  \draw[line width=0.7pt, fill = blue!33!green] (1,0)--(2,0)--(2,3)--(1,3)--(1,0);
  \draw[line width=0.7pt, fill = blue!22!green] (2,0)--(3,0)--(3,2.72)--(2,2.72)--(2,0);
  \draw[line width=0.7pt, fill = blue!11!green] (3,0)--(4,0)--(4,2.36)--(3,2.36)--(3,0);
  \draw[line width=0.7pt, fill = blue!0!green] (4,0)--(5,0)--(5,2.00)--(4,2.00)--(4,0);
  \draw[line width=0.7pt] (0.1,2)--(-0.1,2) node[left]{0.5};
  \draw[line width=0.7pt] (0.1,4)--(-0.1,4) node[left]{1};
  \node[rectangle] at (-0.8, 4.8) {\small{$\frac{\textnormal{time(BACA)}}{\textnormal{time(ACA)}}$}};
  \node[rectangle] at (0.5, 4.2) {\small{1.00}};
  \node[rectangle] at (1.5, 3.2) {\small{0.75}};
  \node[rectangle] at (2.5, 2.92) {\small{0.68}};
  \node[rectangle] at (3.5, 2.56) {\small{0.59}};
  \node[rectangle] at (4.5, 2.2) {\small{0.50}};
  \node[rectangle] at (0.5, -0.3) {\small{ACA}};
  \node[rectangle] at (1.5, -0.3) {\small{$p_{1}$}};
  \node[rectangle] at (2.5, -0.3) {\small{$p_{2}$}};
  \node[rectangle] at (3.5, -0.3) {\small{$p_{3}$}};
  \node[rectangle] at (4.5, -0.3) {\small{$p_{4}$}};
  \node[rectangle] at (5.5, -0.3) {\small{$p_{i}$}};
  \end{tikzpicture}
  \caption{Triangulation with 28\,672 triangles} \label{fig:time2}
  \end{minipage}
  \end{figure}
    The two figures show that we obtain a lower time consumption even in the case of $p_1$, where the structural differences in the right-hand side are low. If the
  singularity is moved closer to the boundary of $\Omega$, BACA speeds up continuously. 
  
  One of the biggest advantages of ACA is its linear logarithmic complexity. 
  In order to observe this behaviour, the sphere discretised with 642 points and 1280 triangles is considered.
  We keep the resulting geometry and increase the number of triangles $N$ for the
  boundary value problem with $p_1$.
  The time columns of Tab.~\ref{tab:NumBall} show the expected complexity of ACA and also of the 
  new BACA. The growth factor of $N$ is four and the growth factor of the computational time is between five and six.
  
     \begin{table}[htb] \centering
   \scalebox{1.0}{ 
  \begin{tabular}{  r  | c | c r | c  r }
    $N$     & 	 \multicolumn{3}{c}{BACA}	 & 	\multicolumn{2}{|c}{ACA}		 \\
		& $\norm{b - Ax_k}_2$ 	& $e_h$  &time (s) &$e_h$	& time (s)	\\  \hline 
  1\,280	&1.58e-08	&0.002		&0.27		&0.002	 &0.31 	\\ [2pt]
  5\,120	&1.49e-08	&0.002		&1.46		&0.002	 &1.83	\\ [2pt]
  20\,480	&1.08e-08	&0.003		&7.34		&0.003	 &10.20 \\ [2pt]
  81\,920	&4.22e-09	&0.004		&38.15		&0.004	 &54.15 \\ 
  \end{tabular}}
  \caption{Numerical results of BACA with $p_1$ and~$\varepsilon_{\text{BACA}} = 5 \cdot 10^{-8}$.}  \label{tab:NumBall}
  \end{table}
  
   \subsection*{Effects on partially overrefined surfaces}
  In the following tests we apply ACA and BACA to triangulations which are partially overrefined. The
  boundary of the ellipsoid $\Omega = \{ x \in \mathbb{R}^3 \ : \ x_1^2 + x_2^2 + x_3^2/9 = 1\}$ is more refined on the mid-region than on the remaining 
  geometry; 
  see Fig.~\ref{fig:mesh_ell}.
  
   \begin{figure}[htb] \centering
  \scalebox{0.3}{\includegraphics[angle=90]{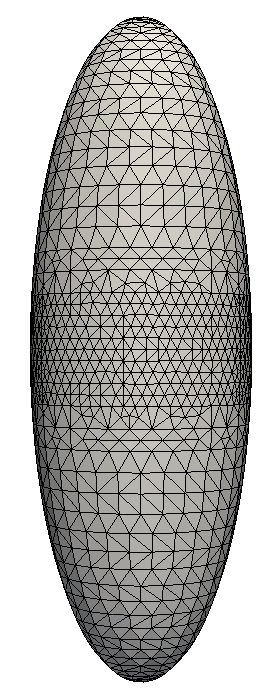}} 
  \scalebox{0.3}{\includegraphics[angle=90]{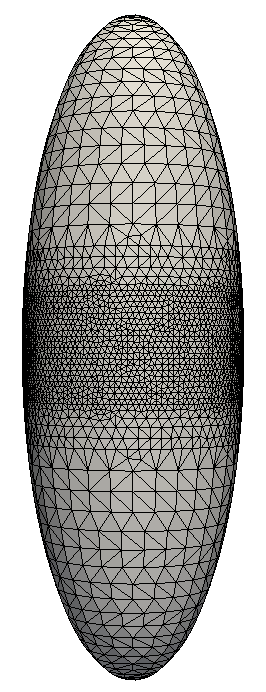}} 
  \scalebox{0.3}{\includegraphics[angle=90]{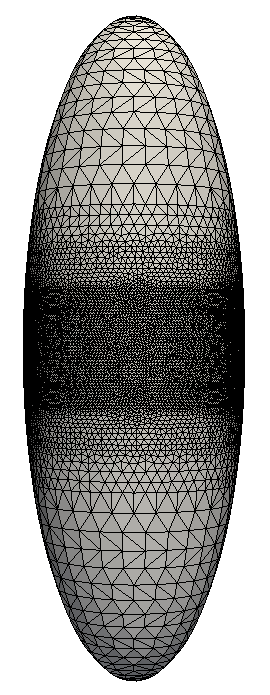}} 
  \scalebox{0.3}{\includegraphics[angle=90]{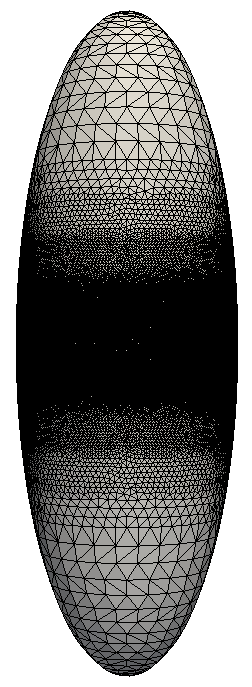}} 
   \caption{Meshes on the considered ellipsoid.} \label{fig:mesh_ell}
  \end{figure}
  
  The numerical results of BACA for the ellipsoid ($p_2$, $\varepsilon_{\text{BACA}} = 10^{-6}$) with a growing number of triangles are contained in 
  Tab.~\ref{tab:NumEll}.  The accuracy of the block approximation for ACA is $\varepsilon_{\text{ACA}} = 10^{-6}$. The minimal block size is also given
  in Tab.~\ref{tab:NumEll}, which we increase with the growing size of the problem. In order to compare the results, we keep the relative error 
  at the same level.
  
  Inspecting the time columns of Tabs.~\ref{tab:NumEll} and~\ref{tab:NumEll_ACA} reveals that BACA is significantly faster than ACA.
  For the largest problem BACA is more than twice as fast as ACA. 
  The reason for the acceleration is that the solution does not benefit from the local over-refinement. While the solver based on ACA is not able to detect this, the 
  combination of solution and error estimation used in BACA leads to lower number of CG iterations and computational time. 
 For the largest problem the number of CG iterations required for the solution based on ACA is more than twice as large as for the solution based on BACA.
  We obtain no significant 
  difference in the storage requirements after the application of the two algorithms. 
  Hence, the observed speed-up results from the coupling of the iterative solver and the error estimator rather than from the reduction of
  the storage requirements. In addition to these facts, a decreasing number of required iterations $k$ can be observed with a growing number of unknowns $N$. 
  
     \begin{table}[htb] \centering
   \scalebox{1.0}{
  \begin{tabular}{  r  c | c c|  c c  r r r  }
  $N$ 		&$b_\textnormal{min}$ &$\varepsilon_{\text{BACA}}$&$\norm{b - Ax_k}_2$& $e_h$	&$k$	&time (s) &CG 		&$\quad$storage	\\
		&	&	&		&	&	& 		&$\quad$iterations	& (MB)		\\ \hline 
  3\,452	&15	&1e-06 	&1.39e-06	&0.034	&10	&1.22		&134			&11.61		\\ 
  8\,574	&30	&1e-06	&1.16e-06	&0.023	&10	&4.89		&273			&42.10		\\ 
  30\,642	&60	&1e-06	&1.04e-06	&0.012	&6	&34.80		&500			&227.07		\\ 
  125\,948	&120	&1e-06	&1.21e-06	&0.006 	&6	&569.36		&1119			&1608.32	\\ 
  \end{tabular}}
  \caption{Numerical results BACA for the ellipsoid.} \label{tab:NumEll}
  \end{table} 
  
     \begin{table}[htb] \centering
   \scalebox{1.0}{
  \begin{tabular}{  r  c | c c  r r r  }
  $N$ 		&$b_\textnormal{min}$ & $e_h$	&$k$	&time (s) &CG 		&$\quad$storage	\\
		&	&	&	& 		&$\quad$iterations	& (MB)		\\ \hline 
  3\,452	&15	&0.034	&10	&1.54		&196			&13.67		\\ 
  8\,574	&30	&0.023	&10	&6.91		&371			&48.37		\\ 
  30\,642	&60	&0.012	&6	&66.63		&909			&257.59		\\ 
  125\,948	&120	&0.006 	&6	&1206.28	&2828			&1674.62	\\ 
  \end{tabular}}
  \caption{Numerical results ACA for the ellipsoid.} \label{tab:NumEll_ACA}
  \end{table}
  
  The ratio of the time required for ACA and for BACA can be seen in Fig.~\ref{fig:TimeQuo}. The curves stabilize at a 
  certain constant, which seems to depend on the geometry and the number and the location of refined areas.
  
  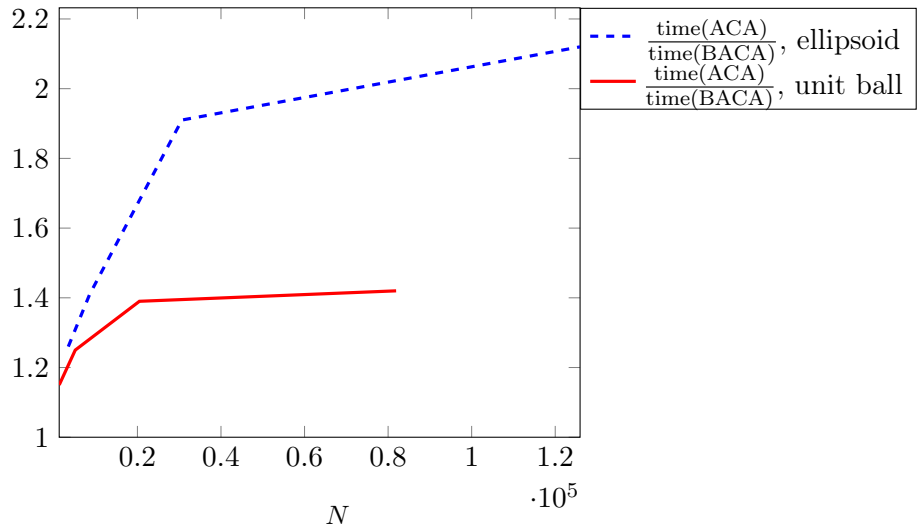
\begin{figure}[htb] \centering \scalebox{1.0}{
  \begin{tikzpicture}
  \begin{axis}[
       xmin = 1280, xmax =125948,
       ymin = 1,
  legend style = {at={(1,1)}, anchor = north west},
  no markers]
  \addplot[blue, very thick, dashed] coordinates{
  (3452,1.26)
  (8574,1.41)
  (30642,1.91)
  (125948,2.12)
  };
  \addlegendentry{$\frac{\text{time(ACA)}}{\text{time(BACA)}}$, ellipsoid}
  \addplot[red, very thick] coordinates{
  (1280, 1.15)
  (5120, 1.25)
  (20480, 1.39)
  (81920, 1.42)
  };
  \addlegendentry{$\frac{\text{time(ACA)}}{\text{time(BACA)}}$, unit ball}
  \end{axis}
  \end{tikzpicture}
  }\vspace*{-0.2cm}\\
  \hspace*{-3.2cm}\small{$N$}\\
  \caption{The ratio of the time required for BACA and for ACA for two geometries.} \label{fig:TimeQuo}
  \end{figure}


\begin{thebibliography}{27}
  \bibitem{ag2017}
  M.~Ainsworth and C.~Glusa.
  \newblock Aspects of an adaptive finite element method for the fractional Laplacian: a priori error estimates, efficient implementation and multigrid solver.
  \newblock {\em Comput. Methods Appl. Mech. Eng},  327:4--35, 2017.

  \bibitem{affkd13}
  M.~Aurada, M.~Feischl, T.~F\"uhrer, M.~Karkulik, and D.~Praetorius.
  \newblock Efficiency and optimality of some weighted-residual error estimator for adaptive 2D boundary element methods.
  \newblock {\em Comput. Methods Appl. Math.}, 13(3):305--332, 2013.
  
  \bibitem{afl12}
  A.~Aurada, S.~Ferraz-Leite, and D.~Praetorius.
  \newblock Estimator reduction and convergence of adaptive BEM.
  \newblock {\em Appl. Numer. Math.}, 62(6):787--801, 2012.

  \bibitem{bebendorf00}
  M.~Bebendorf.
  \newblock Approximation of boundary element matrices.
  \newblock {\em Numer. Math.}, 86(4):565--589, 2000.
  
  \bibitem{bebendorf05}
  M.~Bebendorf.
  \newblock Efficient inversion of Galerkin matrices of general second-order elliptic differential operators with nonsmooth coefficients.
  \newblock {\em Math. Comp.}, 74(251):1179--1199, 2005.
  
  \bibitem{bebendorf05/2}
  M.~Bebendorf.
  \newblock Hierarchical LU decomposition based preconditioners for BEM.
  \newblock {\em Computing}, 74(3):225--247, 2005.
  
  \bibitem{bebendorf06}
  M.~Bebendorf,
  \newblock Approximate inverse preconditioning of finite element discretisations of elliptic operators with nonsmooth coefficients.
  \newblock {\em SIAM J. Matrix Anal. Appl.}, 27(4):909--929, 2006.
  
  \bibitem{bebendorf08}
  M.~Bebendorf.
  \newblock Hierarchical Matrices: A Means to Efficiently Solve Elliptic Boundary Value Problems.
  \newblock Volume 63 of {\em Lecture Notes in Computational Science and Engineering (LNCSE)}, Springer, Berlin, 2008.

  \bibitem{MBMBMB13}
  M.~Bebendorf, M.~Bollh\"ofer, and M.~Bratsch.
  \newblock On the spectral equivalence of hierarchical matrix preconditioners for elliptic problems.
  \newblock {\em Math. Comp.}, 85(302):2839--2861, 2016.
  
  \bibitem{bg06}
  M.~Bebendorf and R.~Grzhibovskis.
  \newblock Accelerating Galerkin BEM for linear elasticity using adaptive cross approximation.
  \newblock {\em Math. Meth. Appl. Sci.}, 29(4):1721--1747, 2006.
  
  \bibitem{BeKr04}
  M.~Bebendorf and R.~Kriemann,
  \newblock Fast parallel solution of boundary integral equations and related problems.
  \newblock {\em Comput. Visual. Sci.}, 8(3-4):121--135, 2005.

  \bibitem{MBYMBS13}
  M.~Bebendorf, Y.~Maday, and B.~Stamm.
  \newblock Comparison of some reduced representation approximations.
  \newblock In A.~Quarteroni and G.~Rozza, editors, {\em Reduced Order Methods
  for modeling and computational reduction}, volume~8 of {\em Springer MS\&A
  series}, pages 67--100. 2014.
  
  \bibitem{br03}
  M.~Bebendorf and S.~Rjasanow,
  \newblock Adaptive low-rank approximation of collocation matrices.
  \newblock {\em Computing}, 70(1):1--24, 2003.
  
  \bibitem{npv12}
  E.~Di Nezza, G.~Palatucci, and E.~Valdinoci.
  \newblock Hitchhiker's guide to the fractional Sobolev spaces.
  \newblock {\em Bull. Sci. math.}, 136(5):521--573, 2012.

  \bibitem{dg13}
  M.~D'Elia and M.~Gunzburger.
  \newblock The fractional Laplacian operator on bounded domains as a special case of the nonlocal diffusion operator.
  \newblock {\em Comput. Math. Appl.}, 66(7):1245--1260., 2013,

  \bibitem{doerfler96}
  W.~D\"orfler.
  \newblock A convergent adaptive algorithm for Poisson's equation.
  \newblock {\em SIAM J. Numer. Anal.}, 33(3):1106--1124, 1996.

  \bibitem{fp08}
  S.~Ferraz-Leite and D.~Praetorius.
  \newblock Simple a posteriori error estimators for the h-version of the boundary element method.
  \newblock {\em Computing}, 83(4):135--162, 2008.

  \bibitem{g13}
  T.~Gantumur.
  \newblock Adaptive boundary element methods with convergence rates.
  \newblock{\em Numer. Math.}, 124(3):471--516, 2013.
  
  \bibitem{gh03}
  L.~Grasedyck and W.~Hackbusch.
  \newblock Constructions and arithmetics of $\Hm$-matrices.
  \newblock {\em Computing}, 70(4):295--334, 2003.

  \bibitem{gr87}
  L.F.~Greengard and V.~Rokhlin.
  \newblock A fast algorithm for particle simulations.
  \newblock {\em J. Comput. Phys.}, 73(2):325--348, 1987.
  
  \bibitem{hackbusch99}
  W.~Hackbusch.
  \newblock A sparse matrix arithmetic based on $\Hm$-matrices. Part I: Introduction to $\Hm$-matrices.
  \newblock {\em Computing}, 62(2):89--108, 1999.
  
  \bibitem{hk00}
  W.~Hackbusch and B.N.~Khoromskij.
  \newblock A sparse $\Hm$-matrix arithmetic. Part II: Application to multi-dimensional problems.
  \newblock {\em Computing}, 64(1):21--47, 2000.
  
  \bibitem{kop13}
  M.~Karkulik, G.~Of, and D.~Praetorius.
  \newblock Convergence of adaptive 3D BEM for weakly singular integral equations based on isotropic mesh-refinement.
  \newblock {\em Numerical Methods for Partial Differential Equations}, 29(6):2081--2106, 2013.

  \bibitem{saad03}
  Y.~Saad.
  \newblock Iterative Methods for Sparse Linear Systems.
  \newblock PWS Publishing Company, Boston (1996).
  
  \bibitem{SaSc11}
  S.A.~Sauter and C.~Schwab.
  \newblock Boundary Element Methods.
  \newblock {\em Springer Series in Computational Mathematics 39}, Springer, Berlin Heidelberg, 2011.
 
  \bibitem{os08}
  O.~Steinbach.
  \newblock Numerical Approximation Methods for Elliptic Boundary Value Problems.
  \newblock Springer, New York, 2008.
  
  \bibitem{t96}
  E.E.~Tyrtyshnikov.
  \newblock Mosaic-skeleton approximations.
  \newblock {\em Calcolo}, 33(1-2):47--57, Toeplitz matrices: structures algorithms and applications, Cortona, 1996.

\end{thebibliography}
\end{document}